\documentclass[11pt]{amsart}

\pdfoutput=1

\usepackage[utf8]{inputenc}
\usepackage[T1]{fontenc}
\usepackage{lmodern}
\usepackage{amsthm, amssymb, amsmath, amsfonts, mathrsfs}
\usepackage{aligned-overset}

\usepackage{microtype}

\usepackage[pagebackref,colorlinks=true,pdfpagemode=none,urlcolor=blue,
linkcolor=blue,citecolor=blue]{hyperref}

\usepackage{amsmath,amsfonts,amssymb,amsthm}
\usepackage{amsthm, amssymb, amsmath, amsfonts, mathrsfs}

\usepackage{mathrsfs}
\usepackage{MnSymbol}
\usepackage{scalerel} % for scaling the cube

\usepackage{xcolor}
\usepackage{accents}

\usepackage{bbm}

\definecolor{labelkey}{gray}{.8}
\definecolor{refkey}{gray}{.8}

\definecolor{darkred}{rgb}{0.9,0.1,0.1}
\definecolor{darkgreen}{rgb}{0,0.5,0}

\newtheorem{theorem}{Theorem}[section]
\newtheorem{lemma}[theorem]{Lemma}

\newtheorem{proposition}[theorem]{Proposition}

\theoremstyle{remark}
\newtheorem{remark}[theorem]{Remark}

\renewenvironment{proof}[1][Proof]{ {\itshape \noindent {#1.}} }{$\Box$
\medskip}

\numberwithin{equation}{section}
\newcommand{\R}{\mathbb{R}}

\newcommand{\Z}{\mathbb{Z}}
\newcommand{\N}{\mathbb{N}}
\newcommand{\Pb}{\mathbb{P}}
\newcommand{\PP}{\mathbf{P}}
\newcommand{\E}{\mathbb{E}}

\newcommand{\B}{\mathcal{B}}

\newcommand{\eps}{\varepsilon}

\newcommand{\EE}{\mathbf{E}}

\newcommand{\ONE}{\mathbf{1}}

\begin{document}

\title[Localization length of directed polymer]{Localization length of the  $1+1$  continuum directed random polymer}
\author{Alexander Dunlap\and Yu Gu\and Liying Li}

\address[Alexander Dunlap]{Department of Mathematics, Courant Institute of Mathematical Sciences, New York University, New York, NY 10012 USA}
\address[Yu Gu]{Department of Mathematics, University of Maryland, College Park, MD 20742, USA}
\address[Liying Li]{Department of Mathematics, University of Michigan, Ann Arbor, MI 48109, USA}

\maketitle

\begin{abstract}
In this paper, we study the localization length of the $1+1$  continuum directed polymer, defined as the distance between the endpoints of two paths sampled independently from the quenched polymer measure. We show that the localization length converges in distribution in the thermodynamic limit, and derive an explicit density formula of the limiting distribution.  As a consequence, we prove the $\tfrac32$-power law decay of the density, confirming the physics prediction of Hwa-Fisher \cite{fisher}. Our proof uses the recent result of Das-Zhu \cite{daszhu}.
\bigskip

%\noindent \textsc{MSC 2010:} 		35R60, 60H07, 60H15.

%\medskip

\noindent \textsc{Keywords:} Directed polymer, stochastic heat equation, Brownian bridge.

\end{abstract}
\maketitle

\section{Introduction}
\subsection{Main result}

Consider the stochastic heat equation with a $1+1$ spacetime white noise $\xi$ and the Dirac initial data
\begin{equation}\label{e.she}
\begin{aligned}
&\partial_t u(t,x)=\frac12\Delta u(t,x)+u(t,x)\xi(t,x), \quad t>0, \\
&u(0,x)=\delta(x).
\end{aligned}
\end{equation}
We define the random density 
\begin{equation}\label{e.defrho}
\rho(t,x)= \frac{u(t,x)}{\int_{\R} u(t,x')dx'}, \quad\quad t>0.
\end{equation}
On a formal level, one can write the solution by the Feynman-Kac formula as 
\[
u(t,x)=\E [ \delta(B_t-x)e^{\int_0^t \xi(s,B_s)ds}],
\]
where $B$ is a standard Brownian motion starting from the origin and $\E$ is the expectation on $B$. This can be viewed as the partition function of the point-to-point directed polymer of length $t$. Then $\rho$ in \eqref{e.defrho} is given by the formal expression
\[
\rho(t,x)= \frac{\E [ \delta(B_t-x)e^{\int_0^t \xi(s,B_s)ds}]}{\E [ e^{\int_0^t \xi(s,B_s)ds}]}
\]
and can be viewed as the quenched endpoint density of the polymer path of length $t$. 

The above Feynman-Kac formula is only formal as $\xi$ is a spacetime white noise thereby the expression $\int_0^t \xi(s,B_s)ds$ makes no mathematical sense. There is a rigorous formulation of the so-called $1+1$ continuum directed polymer, which was introduced in \cite{akq}.  As we are mostly interested in the endpoint distribution, rather than the distribution of the whole path, throughout the paper we will just stick to the definition given by \eqref{e.defrho}: given a realization of $\xi$, the endpoint of the polymer path of length $t$ has density $\rho(t,\cdot)$. We use $\hat{\E}_t$ and $\hat{\Pb}_t$ to represent the quenched expectation and probability, so for any bounded $f$ we have $\hat{\E}_t f(B_t)=\int_{\R} f(x)\rho(t,x)dx$,
and for any measurable $A\subseteq \R$, we have $\hat{\Pb}_t (B_t\in A)=\int_A \rho(t,x)dx$.

Let $B_t^1,B_t^2$ be two endpoints sampled independently from $\hat{\Pb}_t$. We define  
\begin{equation}\label{e.defell}
\ell(t)=B_t^1-B_t^2.
\end{equation}
If there were no random environment, then \eqref{e.she} would simply reduce to the standard heat equation, and $B_t^1,B_t^2$ would be  two independent Gaussian random variables with variance $t$, in which case $\ell(t)\sim \sqrt{t}$ for large $t$. With the presence of the random environment, the polymer paths preferentially visit the regions with large values of $\xi$. As a result, the paths localize and one may expect $\ell(t)$ to be of order $O(1)$ for large $t$. Thus, one can view $\ell(t)$ as describing the localization length, and we are interested in the distribution of $\ell(t)$ for large $t$. Throughout the paper, we will use $\EE$ and $\PP$ to denote the expectation and probability with respect to the noise $\xi$. Denote the quenched density of $\ell(t)$ by $g$, i.e., 
\[
g(t,x)=\int_{\R}\rho(t,y)\rho(t,y+x)dy,
\] 
and the annealed density of $\ell(t)$ by 
\[
\bar{g}(t,x)=\EE g(t,x),
\] so for any bounded $f$, $\EE \hat{\E}_t f(\ell(t))=\int_{\R} f(x)\bar{g}(t,x)dx$.

Here is the main result of the paper:
\begin{theorem}\label{t.mainth}
As $t\to\infty$, $\ell(t)$ converges in distribution in the annealed sense, and the limit has an explicit density $\bar{g}_\infty(x)$ given by \eqref{e.exginfinity} below. In other words, for any $f\in C_b(\R)$, we have 
\begin{equation}\label{e.mainconver}
\EE \hat{\E}_t f(\ell_t)=\int_\R f(x)\bar{g}(t,x)dx\to \int_{\R} f(x)\bar{g}_\infty(x)dx
\end{equation}
as $t\to\infty$. In addition, 
\begin{equation}\label{e.tail}
|x|^{\frac{3}{2}}\bar{g}_\infty(x)\to \lambda, \quad\quad \mbox{ as } |x|\to\infty,
\end{equation}
with $\lambda>0$ given by \eqref{e.asyginfinity} below.
\end{theorem}

\subsection{Context}
The directed polymer in a random environment is a prototypical model of disordered systems in non-equilibrium statistical mechanics, and has been studied extensively since the 1980s. Here we will not attempt to review the large body of literature; instead, we refer to the excellent book \cite{comets2017directed} and the references cited there. A salient feature of the model is the localization phenomenon in the so-called low temperature regime, which is the physically interesting regime in which the polymer path is constrained to a small set of states that are energetically favorable. In the high temperature regime, the   paths behave diffusively as the Brownian motion, which is somewhat easier to analyze. When the temperature is low, the paths are expected to behave super-diffusively and at the same time localize around some favorite region. While such behaviors are notoriously difficult to quantify, mathematical studies    have achieved important progress in recent years. This has involved both the study of the fluctuations of the endpoint displacement and the free energy, which falls into the $1+1$ KPZ universality class \cite{amir2011probability,BQS11,bc1,bc,BCF14,BCFV15,bcr,corwin2012kardar,corwinh,haoshen,Qua12,quastel2015one,timo,timo1}, and also the quantitative analysis of localization behaviors \cite{bakhtin2020localization,erik2,bates2020endpoint,erik1,broker2019localization,C2,comets1,comets4,comets2,comets3,daszhu}.

Our study of the localization length defined in \eqref{e.defell} is motivated by the physics work of Hwa-Fisher \cite{fisher}, where the same object was analyzed with a focus on ``anomalous fluctuations'': fluctuations that result from rare events but that nonetheless are expected to play an important role in determining many thermodynamic properties and average correlations. We illustrate this phenomenon through a folklore example below \cite{fisher1}.

Consider any random environment $\xi$ that is statistically invariant under shear transformations, that is, for any $\theta\in\R$,  
\[
	\{\xi(t,x)\}_{t,x}\overset{\text{law}}{=} \{\xi(t,x+\theta t)\}_{t,x}.
\]   One can check that any Gaussian process that is white in time and stationary in space satisfies this condition, and in particular, it applies to the spacetime white noise considered in this paper. Define the tilted free energy $F_\theta(t):=\log \E e^{\theta B_t} e^{\int_0^t \xi(s,B_s)ds}$. Then by the Girsanov theorem and shear invariance, we immediately obtain
\[
\begin{aligned}
F_\theta(t)&=\log \E e^{\theta B_t-\frac12\theta^2 t}e^{\int_0^t \xi(s,B_s)ds}+\frac12\theta^2 t\\
&=\log \E e^{\int_0^t \xi(s,B_s+\theta s)ds}+\frac12\theta^2 t\\
	\overset{\text{law}}&{=} F_0(t)+\frac12\theta^2 t.
\end{aligned}
\]
On the one hand, the quenched variance of the polymer endpoint is 
\[
%:
\mathsf{V}_q(t):=\int_\R x^2 \rho(t,x)dx-(\int_\R x\rho(t,x)dx)^2=\partial_\theta^2 F_\theta(t)\big|_{\theta=0}.
\]
Taking the average on the random environment leads to
\begin{equation}\label{e.qvart}
\begin{aligned}
\EE \mathsf{V}_q(t)&=\EE \partial_\theta^2 F_\theta(t)\big|_{\theta=0}=\partial_\theta^2 \EE F_\theta(t) \big|_{\theta =0}\\
&=\partial_\theta^2 (\EE F_0(t)+\frac12\theta^2 t) \big|_{\theta=0}=t.
\end{aligned}
\end{equation}
On the other hand, since $B_t^1,B_t^2$ are independently sampled from $\hat{\Pb}_t$, we actually have 
\[
\hat{\E}_t |\ell(t)|^2=\hat{\E}_t |B_t^1-B_t^2|^2=2 \mathsf{V}_q(t),
\]
which, combined with \eqref{e.qvart}, implies that $\EE \hat{\E}_t |\ell(t)|^2=2t$. 

Our main result implies that, in the annealed sense, $\ell(t)$ is of order $O(1)$ for $t\gg1$. That is, for any $\eps>0$, there exists an $M<\infty$ such that
\[
\limsup_{t\to\infty}\PP\hat{\Pb}_t [|\ell(t)|<M]>1-\eps.
\] While this is consistent with the localization phenomenon, it also suggests that the blowup of $\EE \hat{\E}_t |\ell(t)|^2$ as $t\to\infty$ indeed results from the rare events in which $\ell(t)$ is large. The fact that the limiting distribution of $\ell(t)$ has a heavy tail  provides a further evidence for this. As a matter of fact, it was conjectured in \cite[Eq. (3.18)]{fisher} that the annealed density of $\ell(t)$, which was denoted by $\bar{g}(t,x)$, takes the following scaling form
\begin{equation}\label{e.ficon}
\bar{g}(t,x)\approx \frac{1}{|x|^{3/2}} w(\frac{x}{t^{2/3}}), \quad\quad \text{ for } |x|\gg1,
\end{equation}
with the scaling function $w$  decaying rapidly at infinity. The derived tail asymptotics in \eqref{e.tail} can be viewed as a mathematical justification of \eqref{e.ficon} at $t=\infty$. 

We do not expect the limiting density $\bar{g}_\infty(\cdot)$ to be universal. In other words, if one considers another polymer model, with a different reference measure or a different random environment, then the corresponding limit is likely to be different. In the convergence of \eqref{e.mainconver}, there is no rescaling in the $x-$variable, so it is a type of local behaviors of $\rho(t,\cdot)$ we were probing, and this is supposed to be model-dependent. Nevertheless, the $\tfrac32$-power law decay of $\bar{g}_\infty(\cdot)$ is expected to be universal and to apply to all models in the universality class. To the best of our knowledge, this is the first mathematical justification of this fact.

It is worth noting that our result does not imply the convergence of $\mathsf{V}_q(t)$ as $t\to\infty$. Although the localization length $\ell(t)$ is of order $O(1)$ as $t\to\infty$, the quenched
variance $\mathsf{V}_q(t)$ may pick up the tail of $\rho(t, \cdot)$ and blow %\alexcomment{I changed ``blows'' to ``blow'' since we are only saying this ``may'' happen, right?} 
 up as $t\to\infty$. % \gucomment{Liying, could you add someting here about the paper with Kostya?}
There are two types  of
rare events in the model, one coming from the tail of $\rho(t,\cdot)$, for each realization of $\xi$, and the other  from the fluctuations of $\xi$.
Our result is on the annealed law of $\ell(t)$. As will become clear later, the power law decay of the limiting density results from the rare fluctuations of $\xi$, see Remark~\ref{r.quench} below.
In \cite{KL22} the simulation of a discrete polymer model showed that the typical quenched variance can grow as $\mathsf{V}_{q}(t) \gtrsim t^c$ for $c>0$, which suggests that  the tail behavior of $\rho(t,\cdot)$ could
be more complicated than the current localization picture in the literature.

It is known that the study of the directed polymer in a random environment can be reduced to that of   a diffusion with a random drift   given by the solution to a stochastic Burgers equation, see e.g. \cite[Theorem 2]{yk}. So what we considered in this paper is equivalent to studying the separation of two passive scalars in the Burgers flow. A similar problem was addressed in \cite{knr}, where the random environment is given by a Gaussian process with long-range correlations. There, different behaviors of the particle displacements and the separation were observed. In a sense, it is a similar phenomenon to that displayed here with the two-particle separation and the single-particle displacement observed on different scales. This perspective has also been taken in the physics literature, see e.g. \cite{sb} and the references therein.

To obtain the limiting distribution of $\ell(t)$, we rely on a recent result of Das-Zhu \cite{daszhu}, where they studied the same model and proved that the random density $\rho(t,\cdot)$ converges as $t\to\infty$ after a random shift, see \eqref{e.daszhu} below. Their result suggests the convergence in law of $\ell(t)$ as well as a probabilistic representation of the limiting distribution  in terms of a two-sided Bessel-3 process, see \eqref{e.ginfinity} below. The main contribution here is to obtain an explicit formula for the density, using which we further derive the $\tfrac32$-power law decay. Our approach is based on an approximation of the Bessel-3 by a Brownian bridge near its maximum, together with the well-known density formula for certain exponential functionals of the Brownian bridge, derived by Yor \cite{yor1992some}. More precisely, a key object to study is the random density defined in \eqref{e.defrhoL} below:
\begin{equation}\label{eq:rhoLintrodef}
\rho_L(x)=\frac{e^{B_L(x)}}{\int_0^L e^{B_L(x')}dx'},
\end{equation}
where $B_L$ is a Brownian bridge connecting $(0,0)$ and $(L,0)$. It turns out that the density of $\ell(t)$ for large $t\gg1$ can be approximated by the two point covariance function of $\rho_L(\cdot)$ with large $L\gg1$, see Proposition~\ref{p.condensity} below. The random density $\rho_L(\cdot)$ itself is a stationary random process and its multipoint covariance function can be computed through Yor's formula, and that is how we obtain the main result of this paper.

The random density $\rho_L(\cdot)$ defined in \eqref{eq:rhoLintrodef} turns out to also be important in the study of other physical models, such as 1D Liouville quantum mechanics, 1D disordered
  supersymmetric quantum mechanics, and Sinai's model of diffusion in a random environment, among others. See for example \cite{matsumoto2005exponential,comtet,ledoussal,kk}. As a matter of fact, the
  two-point covariance function of $\rho_L(\cdot)$ has been computed e.g.\ in \cite{ledoussal} and \cite{comtet} on the level of physics rigor, and the $\tfrac32$-power law decay has also been derived
  in those models, see \cite[Section 3, Eq.\ (14)]{comtet} and \cite[Eq.\ (57)]{ledoussal}.
  %\sout{It is surprising to us that the same type of calculations can be utilized in the study of the directed polymers and the KPZ equation, and it is unclear why the same exponent $3/2$ appears in all
  %  aforementioned models.   According to \cite[Eq.\ (3.18)]{fisher}, the decay rate should be related to the wandering exponent of the directed polymer which extends to all dimensions, so we are    hoping that such a connection
  % we made can further shed some light on the problem in high dimensions.}
  As pointed out in~\cite{comtet,ledoussal,fisher}, contributions to the annealed $\tfrac32$-power law decay come primarily from the rare events on which another local minimum of the
    potential occurs far away from the global minimum but has comparable depth.  When the density takes the
    form~(\ref{eq:rhoLintrodef}), the probability of having such a second valley at distance $x$ is related to the probability of a Bessel-3 process returning to $0$ at time $x$, which gives $x^{-3/2}$.
    However, we believe that that $\tfrac32$-power law will also hold for other polymer models, where the endpoint distribution is not of the form~(\ref{eq:rhoLintrodef}).
In fact, \cite[Eq.\ (3.18)]{fisher} states that in all dimensions, the power-law decay exponent $\nu$ is related to the wandering exponent~$\zeta$ of the directed polymer by $\nu = d+2 - \zeta^{-1}$.
In one dimension, $\nu=\tfrac32$ is a consequence of $\zeta=\tfrac23$ from the KPZ universality. We hope that this connection can shed some light on the wandering exponents in higher dimensions.

The above random density $\rho_L(\cdot)$ can be viewed as the invariant measure of the directed polymer on a cylinder, which is related to the invariant measure of the KPZ equation with a periodic boundary condition. We refer to \cite{barraquand2021steady,BKWW21,CK21} for recent developments on invariant measures for the \emph{open} KPZ equation, in which non-periodic boundary conditions are imposed.

Throughout the paper, we  use $G_t(x)=(2\pi t)^{-1/2} \exp(-x^2/(2t))$ to denote the standard heat kernel.

\subsection*{Acknowledgements}
A.D. was partially supported by the NSF Mathematical Sciences Postdoctoral Research Fellowship Program through grant no.\ DMS-2002118.
Y.G. was partially supported by the NSF through DMS-2203014. We thank the anonymous referee for suggesting some relevant references.

\section{Convergence in law}

Recall that $\ell(t)=B_t^1-B_t^2$, with $B_t^1,B_t^2$ sampled independently from the quenched density $\rho(t,\cdot)$. 
Recall that $g(t,\cdot)$ is the quenched density of $\ell(t)$:
\[
g(t,x)=\int_\R \rho(t,y)\rho(t,y+x)dy.
\]
Recall that $\bar{g}(t,x)$ is the annealed density, so we have
\[
\bar{g}(t,x)=\EE g(t,x).
\] 

An interesting recent result by Das-Zhu \cite[Theorem 1.5]{daszhu} states that, for each $t>0$, almost surely there exists a unique mode of $\rho(t,\cdot)$, denoted by $x_o(t)$, and the following weak convergence on $C(\R)$ holds:
 \begin{equation}\label{e.daszhu}
 \{\rho(t,x_o(t)+x\}_{x\in\R}\Rightarrow \{\rho_\infty(x)\}_{x\in\R} \quad\quad \mbox{ as }t\to\infty.
 \end{equation}
Here $C(\R)$ is endowed with the uniform-on-compact topology, and the limit takes the form
 \[
 \rho_\infty(x)=\frac{e^{-\B(x)}}{\int_\R e^{-\B(x')}dx'},
 \]
  where $\B$ is a two-sided Bessel-3 process.
  
  By a change of variable, one can write 
  \[
  \bar{g}(t,x)=\EE \int_\R \rho(t,y+x_o(t))\rho(t,y+x_o(t)+x)dy.
  \]
  Inspired by \eqref{e.daszhu}, we define
  \begin{equation}\label{e.ginfinity}
  \bar{g}_\infty(x)=\EE \int_\R \rho_\infty(y)\rho_\infty(y+x)dy\in L^1(\R).
  \end{equation}
Therefore, the proof of Theorem~\ref{t.mainth} reduces to   the study of the function $\bar{g}_\infty(\cdot)$ and 
the following  proposition:
\begin{proposition}\label{p.wkcon}
For any $f\in C_b(\R)$, we have as $t\to\infty$, 
\[
\int_{\R}f(x)\bar{g}(t,x)\to \int_{\R} f(x)\bar{g}_\infty(x)dx.
\]
\end{proposition}

%\begin{proposition}\label{p.density}
%We have 
% \[
% \bar{g}_\infty(x)=
% \]
% and as $|x|\to\infty$,  
% \[
% |x|^{3/2}\bar{g}_\infty(x)\to,
% \]
% \end{proposition}

We prove Proposition~\ref{p.wkcon} first, since it is a consequence of \eqref{e.daszhu}.

\begin{proof}[Proof of Proposition~\ref{p.wkcon}]
The tail estimate in \cite[Prop 7.2]{daszhu} (see also the discussion after Prop~7.2 and in Section~1.4) implies that for all $\eps>0$, if we define
\begin{equation*}
 A_{t,\eps,K} = \Big\{  \int_{[-K,K]^c} \rho \big(  t , s + x_o(t) \big) \, ds > \eps \Big\},
\end{equation*}
then
\begin{equation}
  \label{e.tail-esti}
  \limsup_{K \to \infty} \limsup_{t \to \infty} \PP( A_{t,\eps,K} ) = 0.
\end{equation}
Hence, for all $\eps,\delta>0$, we have $\PP (A_{t,\eps,K}) < \delta$ if  $K$ and $t$ are sufficiently large.

Without loss of generality let us assume that $|f(x)| \le 1$.  
Writing $q_t(x,y) = \rho \big( t, y+x_o(t) \big) \rho \big(t,x+y+x_o(t) \big)$, ${q_\infty(x,y) = \rho_{\infty}(y) \rho_{\infty}(x+y)}$ and $D_L=[-L,L]^2$, we have
\begin{align*}
  & \Big| \int_{\R} f(x) \bar{g}(t,x) \,dx - \int_{\R} f(x) \bar{g}_{\infty}(x) \, dx
  \Big|\\
  & \quad= \Big| \EE \int_{\R^2} f(x)q_t(x,y) \, dxdy - \EE\int_{\R^2} f(x) q_{\infty}(x,y) \, dxdy    \Big|
  \\
  &\quad\le \Big| \EE \int_{D_{2K}} f(x)q_t(x,y) \, dxdy - \EE \int_{D_{2K}} f(x) q_{\infty}(x,y) \, dxdy    \Big|
    \\&\qquad+ \int_{D_{2K}^c} \EE q_{\infty}(x,y) \, dxdy \\
&\qquad  + \EE \ONE_{A_{K,\eps}} \int_{D_{2K}^c}  q_t(x,y)\, dxdy
   + \EE \ONE_{A_{K,\eps}^c} \int_{D_{2K}^c}  q_t(x,y) \,dx dy.
 %  &\le  \Big| \EE \int_{B_{2K}} f(x) q_t(x,y) \, dx dy - \int_{B_{2K}} f(x) q_{\infty}(x,y)\,dxdy   \Big|
 %    + 
 %  &\le
 %    \delta + \limsup_{t \to \infty} \Big|\EE \ONE_{A_{t,\eps,K}^c}
 %    \int_{\R\times \R} f(x) \big(  q_t(x,y) -q_{\infty}(x,y)\big)    \, dxdy \Big|
 %  \\
 % & \le \delta + \limsup_{t \to \infty}\Big| \EE \ONE_{A_{t,\eps,K}^c} \int_{[-2K,2K]^2} f(x) \big(  q_t(x,y) -q_{\infty}(x,y)\big)  \, dx dy \Big| \\
 % & \qquad + \limsup_{t \to \infty} \EE \ONE_{A^c_{t,\eps,K}} \int_{\big( [-2K,2K]^2 \big)^c} q_t(x,y) + q_{\infty}(x,y) \, dx dy.
\end{align*}
 By the weak convergence~(\ref{e.daszhu}), the first term has limit $0$ as $t \to \infty$, since $\rho \mapsto \int_{D_{2K}} f(x) \rho(y) \rho(x+y) \,dxdy$ is a bounded continuous
 functional on $C(\R)$ with the uniform-on-compact
topology.
The second term can be made less than $\eps$ by choosing $K$ large since $\int_{\R^2} \EE q_{\infty}(x,y)\, dxdy = \|\bar{g}_{\infty}\|_{L^1(\R)}<\infty$.
The third term is bounded by $\PP(A_{t,\eps,K}) \le \delta$.
The last term is bounded by $2\eps$ since 
\begin{align*}
&\int_{D_{2K}^c}  q_t(x,y) \,dx dy = \int_{|z-y| \ge 2K \text{ or } |y| \ge 2K}  \rho \big( t, y+x_o(t) \big) \rho \big(t,z+x_o(t) \big) \, dy dz\\
&  \quad\le \int_{|y| \ge K \text{ or } |z| \ge K}  \rho \big( t, y+x_o(t) \big) \rho \big(t,z+x_o(t) \big) \,dydz \le 2\eps
\end{align*}
on the event $A_{t,\eps,K}^c$.
Therefore,
\begin{equation*}
\limsup_{t \to \infty} \Big| \int_{\R} f(x) \bar{g}(t,x) \,dx - \int_{\R} f(x) \bar{g}_{\infty}(x) \, dx
  \Big| \le \eps+ \delta +  2\eps
\end{equation*}
and the proposition follows from choosing $\eps$ and $\delta$ arbitrarily small.
\end{proof}

\section{Derivation of the density formula}

To obtain the explicit formula for $\bar{g}_\infty(x)$ defined in \eqref{e.ginfinity}  is the main technical challenge. There might be a direct way of studying exponential functionals of Bessel-3 through which one finds the formula of $\bar{g}_\infty(\cdot)$. The idea we adopt here is to approximate the Bessel-3 process by  a Brownian bridge near its maximum  and use the well-known formula for the Brownian bridge derived by Yor.

\subsection{Approximation by Brownian bridge}

Define 
\begin{equation}\label{e.defrhoL}
\rho_L(x)=\frac{e^{B_L(x)}}{\int_0^L e^{B_L(x')}dx'}, \quad\quad x\in[0,L],
\end{equation}
where $B_L(\cdot)$ is a Brownian bridge from $(0,0)$ to $(L,0)$. We extend   $B_L(\cdot)$ periodically so that it is defined on $\R$.   Let $x_L$ be the (almost-surely) unique mode  of $B_L(\cdot)$ in $[0,L]$, i.e. \[
x_L = \operatorname*{argmax}\limits_{x\in [0,L]}B_L(x).
\] 

We have the following proposition.
\begin{proposition}\label{p.conrho}
For any $M>0$, we have 
\begin{equation}
\{\rho_L(x_L+x)\}_{x\in[-L/2,L/2]}\Rightarrow \{\rho_\infty(x)\}_{x\in\R}\label{eq:conrho}
\end{equation}
in distribution on $C([-M,M])$, as $L\to\infty$.
\end{proposition}

\begin{remark}
Since the Brownian bridge is the unique invariant measure for the KPZ equation on the torus, the law of $\rho_L$ is actually the unique invariant measure for the endpoint distribution of the directed polymer on a cylinder. Thus, the above result can be interpreted as follows. Suppose we start the stochastic heat equation on a torus of size $L$, with the  Dirac initial data. On one hand, if we first let $L\to\infty$, for each $t>0$, the endpoint distribution of the polymer on the cylinder converges to the endpoint density on the whole space, and the result of Das-Zhu \eqref{e.daszhu} further leads to the convergence of the density as $t\to\infty$ after a random centering. On the other hand, if we let $t\to\infty$ first, the endpoint distribution of the  directed polymer on the cylinder converges to $\rho_L(\cdot)$, then we recenter $\rho_L$ around its mode and send $L\to\infty$, the above proposition shows that we obtain the same limit. In other words, the two limits of $t\to\infty$ and $L\to\infty$ commute. From this perspective, the approximation by Brownian bridge is somewhat natural.
 \end{remark}

To prove the above result, we use the well-known fact that a Brownian bridge near the maximum is a Bessel-3, stated in the following form:
\begin{lemma}\label{l.wkconBBB}
For any $M>0$, we have 
\[
\{B_L(x_L+x)-B_L(x_L)\}_{x\in[-L/2,L/2]}\Rightarrow \{-\B(x)\}_{x\in\R}
\]
in distribution on $C([-M,M])$, as $L\to\infty$.
\end{lemma}

\begin{proof}
This is well-known, but for the convenience of the reader, we present a short proof here.
We consider a Brownian excursion $\{A_L(x)\}_{x\in [0,L]}$ on $[0,L]$. (Informally speaking this is a Brownian bridge on $[0,L]$, with $A_L(0)=A_L(L)=0$, conditioned to be positive on $(0,L)$.) We note (see e.g. \cite{yor-zambotti})  that there is a $3$-dimensional Brownian bridge $\{J_L(x)\}_{x\in [0,L]}$ such that $A_L(x) = |J_L(x)|$ for all $x\in [0,L]$. Now we extend $A_L$ and $J_L$ to be defined on all of $\R$ by periodicity. By the Vervaat transform \cite{vervaat}, the law of $\{B_L(x_L+x)-B_L(x_L)\}_{x\in [-L/2,L/2]}$ is the same as the law of $\{-|J_L(x)|\}_{x\in [-L/2,L/2]}$. Since the norm of a $3$-dimensional Brownian motion has the same law as a Bessel-$3$ process, it suffices to show that $\{J_L(x)\}_{x\in [-L/2,L/2]}$ converges in law to a two-sided $3$-dimensional Brownian motion as $L\to\infty$.

To prove this, we note that there is a $3$-dimensional standard Brownian motion $\{N(x)\}_{x\in [0,L]}$ (started at $N(0)=0$) such that $J_L(x) = N(x)-\frac{x}{L} N(L)$ for all $x\in [0,L]$. This also means that (as long as $L\ge 2M$), for $x\in [-M,0]$, we have \[J_L(x) = J_L(x+L) = N(x+L)-\frac{x+L}{L} N(L) = N(x+L)-N(L)-\frac{x}{L} N(L).\] Now, by the independence of the increments of Brownian motion, if we continue to assume that $L\ge 2M$ then the process given by
\[\tilde N(x) = \begin{cases}N(x)&x\ge 0,\\N(x+L)-N(L)&x<0\end{cases}\] agrees in law with a standard two-sided $3$-dimensional Brownian motion when restricted to $[-M,M]$. Also, the above two representations show that $J_L$ converges uniformly to $\tilde N$ in $C([-M,M])$ as $L\to\infty$, since $x\mapsto\frac xL N(L)$ converges to zero in probability in $C([0,M])$ as $L\to\infty$. This completes the proof.
\end{proof}

%For the convenience of readers, we present a proof of Lemma~\ref{l.wkconBBB} in the appendix.

Now we present a proof of Proposition~\ref{p.conrho}.

\begin{proof}[Proof of Proposition~\ref{p.conrho}]
We start by observing that
\[
 \rho_L(x_L+x)= \frac{e^{B_L(x_L+x)}}{\int_0^L e^{B_L(x_L+x')}dx'} = \frac{e^{B_L(x_L+x)-B_L(x_L)}}{\int_0^L e^{B_L(x_L+x')-B_L(x_L)}dx'}.
\]
Fix $M>0$ and let $M'\ge M$. By Lemma~\ref{l.wkconBBB}, for each $L>0$ we have a standard two-sided Bessel-$3$ process $\B_L$, defined on the same probability space as $B_L$, such that
\begin{equation}\label{eq:coupling}
 \|B_L(\cdot+x_L)-B_L(x_L)+\B_L(\cdot)\|_{C([-M',M'])}\to 0
\end{equation}
in probability as $L\to\infty$. We note that, for $L\ge 2M'$, we have
\begin{align}\int_0^L e^{B_L(x_L+x')-B_L(x_L)}dx'&=\int_{x_L-M'}^{x_L+M'} e^{B_L(x_L+x')-B_L(x_L)}dx'\notag\\&\qquad+\int_{x_L+M'}^{x_L-M'+L} e^{B_L(x_L+x')-B_L(x_L)}dx'\notag\\&=: I_{1;L,M'}+I_{2;L,M'}.\label{eq:splitintegral}\end{align}
For the first integral, we see by \eqref{eq:coupling} that, for any fixed $M'$, we have
\begin{equation}\label{eq:firstintegral}\left|I_{1;L,M'} -\int_{-M'}^{M'}e^{-\B_L(x')}dx'\right|\to 0\end{equation} in probability as $L\to\infty$.
%We have
%\[\int_{\R\setminus[-M,M]}\EE e^{-\BB_L(x')dx'}dx\]
For the second integral, we write
\begin{equation}
 \EE [I_{2;L,M'}] = \int_{x_L+M'}^{x_L-M'+L}\EE e^{B_L(x_L+x')-B_L(x_L)}dx'.\label{eq:secondintegral}
\end{equation}
To estimate the expectation inside the integral, we use the explicit density for $B_L(x_L)-B_L(x_L+x')$, namely
\begin{align*}\PP(B_L(x_L)-B_L(x_L+x')\in dy)%&=\PP(V(x'/L)\in L^{-1/2}dy)%\\
%&= L^{-1/2}\frac{2(L^{-1/2}y)^2\exp(-\frac{(L^{-1/2}y)^2}{(2(x'/L)(1-x'/L)})}{(2\pi (x'/L)^3(1-x'/L)^3)^{1/2}}\\
= \frac{2y^2\exp(-\frac{y^2}{2\{x'\}_L(L-\{x'\}_L)})}{(2\pi \{x'\}_L^3(L-\{x'\}_L)^3)^{1/2}},\qquad y\ge 0,
%&= L^{-1/2}\frac{2(L^{-1/2}y)^2\exp(-\frac{y^2}{(2(x'/L)(1-x'/L)})}{(2\pi (x'/L)^3(1-x'/L)^3)^{1/2}}
% &=-\frac{2xp_t(x)p_{1-t}'(x)}{tp_1(0)}\\
% &=-\frac{2\sqrt{2\pi}L^{-1/2}yp_{x'/L}(L^{-1/2}y)p_{1-x'/L}'(L^{-1/2}y)}{x'/L}\\
% &= -\frac{2\sqrt{2\pi}L^{-1/2}y(2\pi x'/L)^{-1/2}e^{-y^2/(2x')}p_{1-x'/L}'(L^{-1/2}y)}{x'/L}
\end{align*}
where for any $x\in\R $, we define the unique $\{x\}_L \in [0,L)$  such that $x-\{x\}_L\in L\Z$. (This is a consequence of \cite[(1.3)]{DIM77} and the fact that $B_L(x_L)-B_L(x_L+x')$ is a periodized Brownian excursion.)
This means that%, if we put $\alpha = x'(L-x')$, then
\begin{align*}
 \EE& e^{B_L(x_L+x')-B_L(x_L)}\\ &= \frac{2}{(2\pi)^{1/2}   (\{x'\}_L(L-\{x'\}_L)      )^{3/2}}\int_0^\infty y^2\exp\left(-y-\frac{y^2}{\{x'\}_L(L-\{x\}_L)}\right)dy\\&\le\frac{2}{(2\pi)^{1/2}   (\{x'\}_L(L-\{x'\}_L)      )^{3/2}}\int_0^\infty y^2e^{-y}dy.
\end{align*}
Using this in \eqref{eq:secondintegral}, we see that there is a  constant $C<\infty$, independent of $L$, such that
\[
 \EE[I_{2;L,M'}]\le C(M')^{-1/2}.
\]
Thus by Markov's inequality, for any $\eps_1,\eps_2>0$, we can find $M'$ so large that \[\PP(I_{2;L,M'}>\eps_1/3)<\eps_2/3\] for any $L\ge 2M'$, and also (since the law of $e^{-\B_L(\cdot)}$ does not depend  on $L$ and is almost surely integrable over $\R$) large enough that \[\PP\left(\int_{\R\setminus[-M',M']}e^{-\B_L(x')}dx'>\eps_1/3\right)<\eps_2/3\] for all $L\ge 2M'$. But by \eqref{eq:firstintegral}, if $L$ is large enough (the required largeness depending possibly on $M'$), then \[\PP\left(\left|I_{1;L,M'}-\int_{-M'}^{M'}e^{-\B_L(x')}dx'\right|>\eps_1/3\right)<\eps_2/3.\] Combining the last three displays and recalling \eqref{eq:splitintegral}, we see that
\[\left|\int_0^L e^{B_L(x_L+x')-B_L(x_L)}dx'-\int_\R e^{-\B_L(x')}dx'\right|\to 0\] in probability as $L\to\infty$. Using this along with \eqref{eq:coupling}, we obtain the convergence \eqref{eq:conrho}.
%
%which means that \[\PP\left(\int_0^L e^{B_L(x_L+x')-B_L(x_L)}dx'>\eps_1\right)<\eps_2\] for all sufficiently large $L$.
%
%and we can estimate the integral by
%\begin{align*}
% \int_0^\infty y^2\exp(-y-\frac{y^2}{2\alpha})dy&=\int_0^\infty y^2\exp(-\frac{1}{2\alpha}(y+\alpha)^2+\frac\alpha 2)dy\\
% &=\int_\alpha^\infty (y-\alpha)^2\exp(-\frac{1}{2\alpha}y^2+\frac\alpha 2)dy\\
% &=\int_\alpha^\infty (y-\alpha)^2\exp(-\frac{1}{2\alpha}y^2+\frac\alpha 2)dy
%\end{align*}
%
%
%
%$V(t) = L^{-1/2}[B_L(x_L)-B_L(x_L+Lt)]$
%
%We estimate the expectation inside the integral using the transition density formula proved in \cite[Lemma~4.2]{daszhu}. Specifically, we have
%\begin{align*}
% \EE e^{B_L(x_L+x')-B_L(x_L)} =\int_0^\infty x p_t(y)
%\end{align*}
%
\end{proof}

Using Proposition~\ref{p.conrho}, one can further show
\begin{proposition}\label{p.condensity}
For any $x\in\R$, we have $\bar{g}_\infty(x)<\infty$ and 
\begin{equation}\label{e.condensity}
\begin{aligned}
\bar{g}_\infty(x)&=\lim_{L\to\infty} \int_0^L \EE \rho_L(y+x_L)\rho_L(y+x_L+x)dy\\
&=\lim_{L\to\infty} \int_0^L \EE \rho_L(y)\rho_L(y+x)dy
\end{aligned}
\end{equation}
\end{proposition}

\begin{proof}
First, we show $\bar{g}_\infty(x)<\infty$ for any $x\in \R$:
\[
\begin{aligned}
\bar{g}_\infty(x)=\EE \int_\R \rho_\infty(y)\rho_\infty(y+x)dy &\leq \EE \| \rho_\infty(\cdot)\|_{L^{\infty}(\R)}\\
& =\EE   \frac{1}{\int_\R e^{-\B(x')}dx'} \leq  \EE e^{\sup_{x\in[0,1]} |\B(x)|}<\infty.
\end{aligned}
\]

To simplify the notation, we let ${\hat{\rho}_L (\cdot)= \rho_L(\cdot + x_L)}$. Fix $x\in \R$ and $M > 0$, we first show that %\alexcomment{The ``$=0$'' is not supposed to be there, right?}
\begin{equation}
\label{eq:truncation}
 \lim_{L\to\infty}\EE \int_{-M}^M \hat{\rho}_L(y)\hat{\rho}_L(y+x)dy = \EE \int_{-M}^M \rho_{\infty}(y) \rho_{\infty}(y+x) \, dy .
\end{equation}
For each $K >0$, 
the functional $\rho\mapsto  F_K[\rho]$ given by
\begin{equation*}
F_K[\rho] = \int_{-M}^M (\chi_K \circ \rho)(y) (\chi_K \circ \rho)(x+y)\, dy, \quad \chi_K(p) = \max(p,K),
\end{equation*}
is a bounded continuous functional on $C(\R)$ with the uniform-on-compact topology, so 
by Proposition~\ref{p.conrho},  we have 
\[
\lim\limits_{L \to \infty}\EE F_K [\hat{\rho}_L] = \EE F_K[\rho_{\infty}],
\] 
To finish the proof of (\ref{eq:truncation}), it is enough to show that $\{ \|\hat{\rho}_L\|_{L^{\infty}}^2 \}$ is uniformly integrable, since
\begin{equation*}
\EE \Big( F_{\infty}[\rho] - F_K[\rho]  \Big) \le \EE \ONE_{ \{ \| \rho \|_{L^{\infty}(\R)} \ge K \}} F_{\infty}[\rho] %\le 2M\cdot\EE  \ONE_{\| \rho \|_{L^{\infty}(\R)} \ge K} \| \rho\|^2_{L^{\infty}(\R)}.
\leq 2M\EE  \ONE_{\{\| \rho \|_{L^{\infty}(\R)} \ge K\}} \| \rho\|^2_{L^{\infty}(\R)}.
\end{equation*}
In fact, $\{ B_{L}(x_L) - B_L(x_L+\cdot) \}_{x \in [0,L]}$ has the same distribution as
$\B^0_L$, the Bessel-$3$ bridge over $[0,L]$, which can then be obtained from the standard Bessel-$3$ process via the time change
\begin{equation}
  \label{eq:time-change}
  \B_L^0(y)\overset{\mathrm{law}}{=} \Big( 1-\frac{y}{L} \Big)\B
\Big( \frac{y}{1-y/L} \Big).
\end{equation}
Using $\B_L^0$, we have 
\[
\hat{\rho}_L(y)=\frac{e^{B_L(x_L+y)-B_L(x_L)}}{\int_0^L e^{B_L(x_L+y')-B_L(x_L)}dy'} \overset{\mathrm{law}}{=}\frac{e^{-\B_L^0(y)}}{\int_0^L e^{-\B_L^0(y')}dy'},  %\leq e^{\max\limits_{[0,1]}(B_L(x_L)-B_L(x_L+y'))}
\]
 hence, for $L \ge 2$ and $p \ge 0$,
\begin{align*}
  \EE \|\hat{\rho}_L\|_{L^{\infty}}^p \le \EE e^{p \max\limits_{[0,1]} \B_L^0(y)} = \EE e^{p \max\limits_{[0,1]} (1-y/L)\B(y/(1-y/L))}
  \le \EE e^{p \max\limits_{[0,2]}\B(y) } < \infty.
\end{align*}
This established (\ref{eq:truncation}).

It remains to show that 
\begin{equation*}
\adjustlimits\limsup_{M \to \infty} \limsup_{L \to \infty} \EE \int_{M \le |y| \le L/2} \hat{\rho}_L(y) \hat{\rho}_L(y+x) \, dy = 0.
\end{equation*}
We will only estimate the integral over $M \le y \le L/2$.  The other part is similar.
Using the time change (\ref{eq:time-change}) again, we have
\begin{equation}
\label{eq:estimate-1}
\begin{split}
&\EE \int_{M \le y \le L/2} \hat{\rho}_L(y) \hat{\rho}_L(y+x) \, dy \\
  &\qquad= \EE ( \int_0^L e^{- \B_L^0(y) } \, dy )^{-2} \int_{M \le y \le L/2} e^{- \B_L^0(y) - \B_L^0 (y+x)} \, dy  \\
&\qquad\le \EE e^{2 \max\limits_{0 \le y \le 1} \B_L^0(y)}  \int_{M \le y \le L/2} e^{- \B_L^0(y) } \, dy \\
&\qquad\le \EE e^{2 \max\limits_{0 \le y \le 2} \B(y)} \int_{M \le y \le L/2} e^{- \Big( 1-\frac{y}{L} \Big)\B \Big( \frac{y}{1-y/L}  \Big)} \, dy \\
  &\qquad= \EE e^{2 \max\limits_{0 \le y \le 2} \B(y)} \int_{\frac{LM}{L-M}}^L \frac{L^2}{(t+L)^2}e^{- \frac{L}{t+L}\B (t)} \, dt \\
  &\qquad\le  \EE e^{2 \max\limits_{0 \le y \le 2} \B(y)} \int_{M}^{+\infty} e^{- \frac{1}{2}\B (t)} \, dt,
\end{split}
\end{equation}
where we use the change of variables $y = \frac{tL}{t+L}$ in the penultimate line.

We recall that $\B$ is a Markov process with the transition kernel given by
\begin{equation*}
  p(0,0; t, u) = 2 \frac{u^2}{t} G_t(u), \quad
  p(s,u;t,v ) = \frac{v}{u} \big( G_{t-s}( v-u) - G_{t-s}(v+u)  \big),
\end{equation*}
where $G_t(x)=(2\pi t)^{-1/2}\exp(-x^2/(2t))$.
By writing the Bessel-$3$ process as the modulus of the standard 3D Brownian motion, it is not hard to see that for some constant $K>0$, 
\[
\EE[e^{2 \max_{0 \le y \le 2} \B(y)} | \B(2) = u ] \le K
e^{Ku}.
\]
After conditioning on the value of $\B(2)$ and using the above estimate, the last line of~(\ref{eq:estimate-1}) can be bounded by (up to some multiplicative constant)
\begin{equation}
\label{eq:1}
\begin{split}
& \int_M^{\infty} dt \int_0^{\infty} \, du \int_0^{\infty}  dv \, ue^{-\frac{u^2}{4}}e^{Ku} \frac{ve^{-\frac{v}{2}}}{\sqrt{t-2}}e^{-\frac{u^2+v^2}{2(t-2)}} \Big(
  e^{\frac{2uv}{2(t-2)}} - e^{-\frac{2uv}{2(t-2)}} \Big) \\
&\quad\le C \int_M^{\infty} dt \, \frac{1}{\sqrt{t}}\int_0^{\infty} ve^{-\frac{v}{2}-\frac{v^2}{2(t-2)}} dv \int_0^{\infty} du \, u e^{-\frac{u^2}{8} - \frac{u^2}{2(t-2)} } \Big(
  e^{\frac{2uv}{2(t-2)}} - e^{-\frac{2uv}{2(t-2)}} \Big).
\end{split}
\end{equation}
A direct computation gives 
\begin{align*}
  \int_0^{\infty} ue^{-au^2} (e^{bu} - e^{-bu}) du &=e^{\frac{b^2}{4a}}
  \Big[  \int_0^{\infty} u e^{-a(u-b/2a)^2} \, du - \int_0^{\infty} u e^{-a(u+b/2a)^2} \, du \Big]\\
&  \leq \frac{b}{a} e^{\frac{b^2}{4a}} \int_{-\infty}^{\infty} e^{-au^2}\, du = \frac{\sqrt{\pi}b}{a^{3/2}} e^{\frac{b^2}{4a}}.
\end{align*}
We apply the above computation to (\ref{eq:1}) with $a = \frac{1}{8} + \frac{1}{2(t-2)} = \frac{t+2}{8(t-2)}$, $b = \frac{v}{t-2}$.
Then the integral in $u$ in (\ref{eq:1}) is bounded by $Cvt^{-1}e^{Cv^2t^{-2}}$ for large $M$, and hence the last line is bounded by $CM^{-1/2}$.
This completes the proof.
\end{proof}

\subsection{Yor's formula and exponential functional of Brownian bridge}

In this section, we apply Yor's formula to compute $\EE \rho_L(0)\rho_L(x)$ explicitly. Let us introduce some notations first.  %Recall that $G_t(x)=(2\pi t)^{-1/2}\exp(-x^2/(2t))$. 
 Following \cite[Section 6]{yor1992some}, define 
\begin{equation}\label{e.deftheta}
\theta_r(t):=\frac{r}{(2\pi^3 t)^{1/2}} e^{\frac{\pi^2}{2t}}\int_0^\infty e^{-\frac{w^2}{2t}}e^{-r\cosh w} \sinh w \sin (\tfrac{\pi w}{t}) dw, \quad\quad r>0, t>0,
\end{equation}
and
\begin{equation}\label{e.defa}
a_t(y;x):=\sqrt{2\pi t} e^{\frac{x^2}{2t}} y^{-1} e^{-\frac{1}{2y}(1+e^{2x})} \theta_{e^x/y}(t), \quad\quad x\in\R, y>0.
\end{equation}
Let $\{W_t\}_{t\geq0}$ be a standard Brownian motion, then by \cite[Proposition 2]{yor1992some}, $a_t(\cdot\,;\,x)$ is the density of the random variable $\int_0^t e^{2W_s}ds$ conditioning on $W_t=x$. We have the following technical lemma on the asymptotics of $a$:
\begin{lemma}\label{l.a}
For any $x\in\R,y>0$, we have as $t\to\infty$
\begin{equation}\label{e.limita}
ta_t(y;x)\to \mathsf{h}(y,x),
\end{equation}
with 
\begin{equation}\label{e.defh}
\mathsf{h}(y,x):=y^{-2}e^{x-\frac{1}{2y}(1+e^{2x})}\int_0^\infty w e^{-y^{-1}e^x\cosh w} \sinh w dw.
\end{equation}
There exists $C>0$ such that for all $t\geq 1, x\in \R, y>0$, we have 
\begin{equation}\label{e.bounda}
  ta_t(y;x)\leq  Ce^{\frac{x^2}{2t}} y^{-2} e^{x-\frac{1}{2y}(1+e^{2x})}[1+ye^{-x}|x-\log y|].
  \end{equation}
\end{lemma}

\begin{proof}
To derive \eqref{e.limita}, it suffices to use the fact that $\frac{\sin x}{x}\to1$ as $x\to0$ and $|\sin x|\leq |x|$. To show \eqref{e.bounda}, we have for any $r>0$ that 
\[
\begin{aligned}
&\int_0^\infty e^{-r\cosh w} \sinh w |\sin (\tfrac{\pi w}{t})| dw \leq Ct^{-1}\int_0^\infty we^{-r\cosh w} \sinh w dw\\
&\leq Ct^{-1}\int_0^\infty we^{-\frac{r}{2}e^w} e^w dw= Ct^{-1} \int_1^\infty (\log x) e^{-\frac{r}{2} x} dx\\
&=2Ct^{-1}r^{-1}\int_{r/2}^\infty (\log x-\log \tfrac{r}{2}) e^{-x}dx \leq Ct^{-1}[1+r^{-1}|\log r|].
\end{aligned}
\]
The proof is complete.
\end{proof}

For any $\lambda>0$ and $x\in\R$, define $\mathsf{f}_\lambda(\cdot\,;\,x)$ as the density of the random variable $\int_0^1 e^{\lambda W_s}ds$ conditional on $W_1=x$. Through a change of variable and using the scaling property of the Brownian bridge, we have
\begin{equation}\label{e.defflambda}
\mathsf{f}_\lambda(y;x)=\frac{\lambda^2}{4}a_{\frac{\lambda^2}{4}}(\frac{\lambda^2 y}{4};\frac{\lambda x}{2}).
\end{equation}
%Using $\mathsf{f}_\lambda$, we can compute the two-point covariance function of $\rho_L(\cdot)$ explicitly:
\begin{lemma}\label{l.yor}
For any $L>0$ and $x\in(0,L)$, we have 
\[
\begin{aligned}
\EE \rho_L(0)\rho_L(x)=\frac{1}{16} \int_{\R}\int_{\R_+^2} &e^{ y}G_{x(1-x/L)}(y)(z_1+z_2)^{-2}\\
&\times a_{\frac{x}{4}}(\frac{z_1}{4}; \frac{y}{2}) a_{\frac{L-x}{4}}(\frac{z_2}{4}; \frac{y}{2}) dydz_1dz_2,
\end{aligned}
\]
\end{lemma}

\begin{proof}
Recall that $\rho_L(x)=e^{B_L(x)}/\int_0^L e^{B_L(x')}dx'$ with $B_L$ a standard Brownian bridge connecting $(0,0)$ and $(L,0)$. Fix any $x\in(0,L)$, we have 
\[
\begin{aligned}
\EE \rho_L(0)\rho_L(x)&=\EE e^{B_L(x)}(\int_0^L e^{B_L(x')}dx')^{-2}\\
&=L^{-2} \EE e^{\lambda B_1(\tilde{x})}(\int_0^1 e^{\lambda B_1(x')}dx')^{-2},
\end{aligned}
\]
where $\lambda=\sqrt{L}$, $\tilde{x}=x/L\in(0,1)$ and in the last ``='' we used the scaling property of Brownian bridge. Further conditioning on $B_1(\tilde{x})=y$, we write the above expression as 
\[
\EE \rho_L(0)\rho_L(x)=L^{-2} \int_\R e^{\lambda y}G_{\tilde{x}(1-\tilde{x})}(y)\EE[ (\int_0^1 e^{\lambda B_1(x')}dx')^{-2}\,|\, B_1(\tilde{x})=y]dy,
\]
since $G_{\tilde{x}(1-\tilde{x})}(\cdot)$ is the density of $B_1(\tilde{x})$. We decompose the integral into two parts:
\[
\int_0^1 e^{\lambda B_1(x')}dx'=\left(\int_0^{\tilde{x}}+\int_{\tilde{x}}^1\right) e^{\lambda B_1(x')}dx'=:I_1+I_2.
\]
Conditional on $B_1(\tilde{x})=y$, $I_1$ and $I_2$  are independent, and from the definition of $\mathsf{f}_\lambda$ and a change of variable, we know that their densities are given by
\[
\frac{1}{\tilde{x}}\mathsf{f}_{\lambda \sqrt{\tilde{x}}}(\frac{\cdot}{\tilde{x}}; \frac{y}{\sqrt{\tilde{x}}})\quad\mbox{ and } \quad 
\frac{1}{1-\tilde{x}}\mathsf{f}_{\lambda \sqrt{1-\tilde{x}}}(\frac{\cdot}{1-\tilde{x}}; \frac{y}{\sqrt{1-\tilde{x}}})
\]
respectively. Therefore, we have 
\[
\begin{aligned}
\EE \rho_L(0)\rho_L(x)=L^{-2} \int_{\R}\int_{\R_+^2} &e^{\lambda y}G_{\tilde{x}(1-\tilde{x})}(y)(z_1+z_2)^{-2}\frac{1}{\tilde{x}(1-\tilde{x})}\\
&\times \mathsf{f}_{\lambda \sqrt{\tilde{x}}}(\frac{z_1}{\tilde{x}}; \frac{y}{\sqrt{\tilde{x}}}) \mathsf{f}_{\lambda \sqrt{1-\tilde{x}}}(\frac{z_2}{1-\tilde{x}}; \frac{y}{\sqrt{1-\tilde{x}}}) dydz_1dz_2,
\end{aligned}
\]
which completes the proof after the change of variable $y\mapsto y/\lambda,z_i\mapsto z_i/L, i=1,2$.
\end{proof}

\subsection{Proof of Theorem~\ref{t.mainth}}

By Proposition~\ref{p.condensity}, we have 
\[
\bar{g}_\infty(x)=\lim_{L\to\infty} \int_0^L \EE \rho_L(y)\rho_L(y+x)dy
\]
Recall that  
\[
\rho_L(x)=\frac{e^{B_L(x)}}{\int_0^L e^{B_L(x')}dx'}
\] 
is periodically extended to $\R$,  so for any $y\in\R$, we have 
\[
\rho_L(y+x)=\frac{e^{B_L(y+x)-B_L(y)}}{\int_0^{L} e^{B_L(y+x')-B_L(y)}dx'}.
\]
By a property of Brownian bridge, we know that $\{B_L(y+x)-B_L(y)\}_{x\in[0,L]}$ is again a Brownian bridge connecting $(0,0)$ and $(L,0)$. This implies that $\rho_L(\cdot)$ is a stationary process, which leads to
\[
\int_0^L \EE \rho_L(y)\rho_L(y+x)dy=
L\EE \rho_L(0)\rho_L(x).
\] Applying the following lemma, we conclude the proof of the main theorem.

\begin{lemma}
For any $x>0$, we have 
\begin{equation}\label{e.exginfinity}
\bar{g}_\infty(x)= \frac12\int_{\R}\int_{\R_+^2} e^{2 y}G_{x}(2y)(z_1+z_2)^{-2}
 a_{\frac{x}{4}}(z_1;y)\mathsf{h}(z_2,y)dydz_1dz_2.
\end{equation}
where $a,\mathsf{h}$ were defined in \eqref{e.defa} and \eqref{e.defh} respectively.
%\begin{equation}\label{e.defh}
%\mathsf{h}(z,y):=\frac{16}{z^2}e^{\frac{y}{2}-\frac{2}{z}(1+e^y)}\int_0^\infty  we^{-4e^{y/2}z^{-1}\cosh w} \sinh w  dw, \quad\quad z>0, y\in\R.
%\end{equation}
As $x\to\infty$, we have $x^{\frac{3}{2}} \bar{g}_\infty(x)\to \lambda$ with 
\begin{equation}\label{e.asyginfinity}
\lambda:= \sqrt{\frac{2}{\pi}} \int_{\R}\int_{\R_+^2} e^{ 2y} (z_1+z_2)^{-2}\mathsf{h}(z_1,y) \mathsf{h}(z_2,y)dydz_1dz_2\in(0,\infty). 
\end{equation}
\end{lemma}

\begin{proof}
%First, from \eqref{e.defflambda}, we have
%\[
%\mathsf{f}_{\sqrt{L-x}}(\frac{z_2}{L-x}; \frac{y}{\sqrt{L-x}})=\frac{L-x}{4}a_{\frac{L-x}{4}}(\frac{z_2}{4};\frac{y}{2}),
%\]
%which, combined with 
By Lemma~\ref{l.yor}, we have
\[
\begin{aligned}
L\EE \rho_L(0)\rho_L(x)=\frac14\int_{\R}\int_{\R_+^2} &e^{ y}G_{x(1-x/L)}(y)(z_1+z_2)^{-2}\\
&\times a_{\frac{x}{4}}(\frac{z_1}{4}; \frac{y}{2}) \frac{L}{4}a_{\frac{L-x}{4}}(\frac{z_2}{4}; \frac{y}{2}) dydz_1dz_2,
\end{aligned}
\]
To pass to the limit of $L\to\infty$, we apply \eqref{e.limita} to derive that, 
%it suffices to derive the limit of the term $\frac{L-x}{4}a_{\frac{L-x}{4}}(\frac{z_2}{4};\frac{y}{2})$. From \eqref{e.defa}, we have
%\[
%\frac{L-x}{4}a_{\frac{L-x}{4}}(\frac{z_2}{4};\frac{y}{2})=\sqrt{\frac{\pi}{2}}\frac{(L-x)^{3/2}}{z_2} e^{\frac{y^2}{2(L-x)}-\frac{2}{z_2}(1+e^{y})}\theta_{4e^{y/2}z_2^{-1}}(\frac{L-x}{4}).
%\]
%By the expression of $\theta$ in \eqref{e.deftheta}, we know that for any $r>0$, as $t\to\infty$, 
%\[
%t^{3/2}\theta_r(t)\to\frac{ r}{(2\pi )^{1/2}} \int_0^\infty  ze^{-r\cosh z} \sinh z  dz.
%\]
%Therefore, 
for any fixed $x,z_2>0$ and $y\in\R$, as $L\to\infty$,   
\[
\frac{L}{4}a_{\frac{L-x}{4}}(\frac{z_2}{4};\frac{y}{2})\to %\frac{16}{z_2^2}e^{\frac{y}{2}-\frac{2}{z_2}(1+e^y)}\int_0^\infty  we^{-4e^{y/2}z_2^{-1}\cosh w} \sinh w  dw=
\mathsf{h}(\frac{z_2}{4},\frac{y}{2}).
\]
%We also have 
%\[
%\sup_{t>0}|t^{3/2} e^{-\frac{\pi^2}{2t}}\theta_r(t)| \leq \frac{ r}{(2\pi )^{1/2}} \int_0^\infty  ze^{-r\cosh z} \sinh z  dz,
%\]
%using which one can 
To justify the interchange of limit and integration, it suffices to apply \eqref{e.bounda}, and this completes the proof of \eqref{e.exginfinity}. To prove \eqref{e.asyginfinity}, by the same discussion as before, we have 
\[
\begin{aligned}
\frac{x}{4}a_{\frac{x}{4}}(\frac{z_1}{4};\frac{y}{2})\to \mathsf{h}(\frac{z_1}{4},\frac{y}{2})
\end{aligned}
\]
as $x\to\infty$, and the interchange of limit and integration can be justified in the same way using \eqref{e.bounda}.
\end{proof}

\begin{remark}\label{r.quench}
Suppose we do not take the expectation with respect to $\xi$ and consider the density function%\alexcomment{Isn't this only true in the limit as $t\to\infty$?}
 %which takes the form
\[
\int_\R \rho_\infty(y)\rho_\infty(y+x)dy.
\]
For each realization of $\B$, this is proportional to
\begin{equation}\label{e.10241}
\int_{\R} e^{-\B(y)}e^{-\B(y+x)}dy.
\end{equation}
For almost every realization of the Bessel-$3$ process, we know that there exists $T(\omega)<\infty$ so that $\B(y)\geq y^{1/4}$ for all $y\geq T(\omega)$. Thus, the quenched density of $\ell(t)$, which is proportional to \eqref{e.10241}, has arbitrarily high-order moments. This implies that the heavy tail of $\bar{g}_\infty(x)=\EE\int_\R \rho_\infty(y)\rho_\infty(y+x)dy$ is a result of the rare fluctuations of $\B$.
\end{remark}

\end{document}